\documentclass{article}
\pdfoutput = 1

\usepackage[T1]{fontenc}        
\usepackage[utf8]{inputenc}
\usepackage[english]{babel}     

\usepackage[defaultlines=4,all]{nowidow}
\usepackage{csquotes}          
\usepackage{scalerel,stackengine}

\usepackage{amsmath,xparse}
\usepackage{amsfonts}
\usepackage{amssymb}
\usepackage{amsthm}

\usepackage[usenames,dvipsnames,table]{xcolor}  
\usepackage{graphicx}
\usepackage{stmaryrd}

\usepackage{lmodern}   
\usepackage{bm}               

\usepackage[inline]{enumitem}
\setlist{itemsep=3pt,topsep=3pt,partopsep=2pt,parsep=2pt,leftmargin=30pt}
\setlist[2]{itemsep=2pt,topsep=0pt,partopsep=2pt,leftmargin=20pt,parsep=2pt}
\setlist[enumerate]{label*=\textup{\arabic*.}}

\SetEnumitemKey{ind}{label=\textup{(\roman*)}}
\SetEnumitemKey{dep}{label=\textup{(\alph*)}}
\SetEnumitemKey{seq}{label=\textup{(\arabic*)}}
\SetEnumitemKey{cases}{label=\textup{(\alph*)}}
\SetEnumitemKey{cond}{label=\textup{(\alph*)}}

\usepackage[
            unicode,
            colorlinks=true,
            linkcolor=blue,
            citecolor=cyan 
            ]{hyperref}

\usepackage[capitalise]{cleveref}     

\usepackage{float}

\usepackage{float}
\usepackage{tikz-cd}
\tikzset{close/.style={near start,outer sep=-2pt}}

\usepackage{tikz}
\usetikzlibrary{automata,positioning,decorations.markings}
\usetikzlibrary{decorations.pathmorphing}
\usetikzlibrary{decorations.pathreplacing,calligraphy}
\usetikzlibrary{arrows.meta}
\usetikzlibrary{shapes.geometric}
\tikzset{dot/.style={circle,fill=black,inner sep=0pt, minimum size=0.8pt}}
\tikzset{smallstate/.style={circle,fill=black,inner sep=0pt, minimum size=3pt}}

\usepackage{tikz-cd}

\usepackage{xifthen}
\usepackage{scalerel}

\usepackage[a4paper,left=3cm,top=3.7cm,right=3cm,bottom=3.1cm]{geometry} 
\usepackage[backend = biber,       
            language = english,
            style    = numeric-comp,   
            giveninits = true,
            isbn = false,
            doi = false,
            sorting = nyt,
            backref=true,
            sortcites
            ]{biblatex}

\DeclareNameAlias{sortname}{last-first}
\renewbibmacro{in:}{%
  \ifentrytype{article}{}{%
  \printtext{\bibstring{in}\intitlepunct}}}
\AtEveryBibitem{
    \clearfield{urlyear}
    \clearfield{urlmonth}
    \clearfield{note}%
    \clearfield{url}
}

\DeclareNameAlias{labelname}{given-family}
\renewrobustcmd*{\bibinitdelim}{\,}

\newtheorem{thm}{Theorem}[section]
\newtheorem{cor}[thm]{Corollary}
\newtheorem{lem}[thm]{Lemma}
\newtheorem{prop}[thm]{Proposition}

\theoremstyle{definition}
\newtheorem{defn}[thm]{Definition}

\theoremstyle{remark}
\newtheorem{rem}[thm]{Remark}
\numberwithin{equation}{section}

\newcommand{\aka}{a.k.a.~}     

\newcommand{\rk}{\mathrm{rk}}

\newcommand{\FF}{\mathbb{F}}

\newcommand{\trivial}{1}                
\newcommand{\normalcl}[1]{\langle \! \langle \hspace{1pt} #1 \hspace{1pt} \rangle \! \rangle}

\newcommand{\ff}{\ensuremath{\mathsf{ff}}}

\newcommand{\leqff}{\leqslant_{\ff}}

\newcommand{\isom}{\simeq}      

\newcommand{\Free}[1][]{%
\ifthenelse{\isempty{#1}}{\mathbb{F}}{\mathbb{F}_{\!#1}}
}
\newcommand{\Fn}{\Free[n]}

\newcommand{\defin}[1]{\emph{#1}\index{#1}}

\newcommand{\gen}[1]{\langle  #1 \rangle}

\newcommand{\normaleq}{ \trianglelefteqslant }

\newcommand{\xonto}[2][]{%
  \xrightarrow[#1]{#2}\mathrel{\mkern-14mu}\rightarrow
}

\newcommand{\onto}{\xonto{\ }}

\tikzset{every state/.style={
    inner sep=1pt,
    minimum size=4pt,
    fill=black
    }}

\stackMath
\newcommand\reallywidehat[1]{%
\savestack{\tmpbox}{\stretchto{%
  \scaleto{%
    \scalerel*[\widthof{\ensuremath{#1}}]{\kern-.6pt\bigwedge\kern-.6pt}%
    {\rule[-\textheight/2]{1ex}{\textheight}}
  }{\textheight}%
}{0.5ex}}%
\stackon[1pt]{#1}{\tmpbox}%
}

\newcommand{\set}[1]{\{  #1  \}}        

\newcommand{\CEPeq}{CEP-equipped} 

\newcounter{mallikacomments}

\newcounter{jordicomments}

\newcounter{enriccomments}

\title{Quotient-saturated groups}

\author{Jordi Delgado$^*$, Mallika Roy$^\dag$ and Enric Ventura$^*$}
\date{\vspace{-2pt}
    $^*$Department of Mathematics, Universitat Politècnica de Catalunya (UPC)\\%
    $^\dag$Department of Mathematics, Universidad del País Vasco (UPV/EHU)\\[3ex]%
    \today
}

\begin{document}

\maketitle

\begin{abstract}
We introduce the new notion of \emph{quotient-saturation} as a measure of the immensity of the quotient structure of a group. We present a sufficient condition for a finitely presented group to be quotient-saturated, and use it to deduce that non-elementary finitely presented subgroups of a hyperbolic group (in particular, non-elementary hyperbolic groups themselves) are quotient-saturated. Finally, we elaborate on the previous results to extend the scope of this property to finitely presented acylindrically hyperbolic groups.
\end{abstract}

\section{Introduction}\label{sec: intro}

The notion of quotient by a normal subgroup is central in the algebraic study of groups, starting from the foundational result that every group is (isomorphic to) a quotient of a free group. From this point of view, we can get a reasonable measure of the internal complexity of a group $G$ by looking at its lattice of quotients: the more quotients $G$ has, or the more intricate its lattice of quotients is, the richer the internal structure of $G$  will be. Of course, these kind of sentences are only at the intuitive level, and one needs to state precise definitions and prove concrete results in order to provide them with real mathematical meaning. One of the contributions in this direction was made by Higman,
introducing the notion of \mbox{SQ-universality}: a group~$G$ is \defin{SQ-universal} if any countable group can be embedded in some quotient of $G$. The classical Higman--Neumann--Neumann theorem states that the free group of rank 2 (and so, any non-abelian free group) is SQ-universal. 

In this paper we propose the new notion of \emph{quotient-saturated} group (see \Cref{def: qs} below), also conveying the idea that `$G$ has lots of quotients', and we provide a broad family of examples of such groups. This notion has been inspired by its dual, that of intersection-saturated group, previously introduced in~\cite{delgado_intersection_2023} by the same authors. 

Let $G$ be a group and let $\rk(G)$ denote its rank (i.e., the minimum cardinality of a generating set for~$G$). Of course, the obvious inequality $\rk(G/N)\leq \rk(G)$ is a first restriction for all quotients of $G$. In particular, any quotient of a finitely generated group is finitely generated as well. No analog of this fact works at the level of relations: a finitely related group (even the free group of rank 2) may have non finitely related quotients, and vice-versa, non finitely related groups may have finitely related quotients (the trivial quotient, for instance). We are interested in studying how rich a group $G$ could be in terms of the finite presentability of its quotients. We will restrict ourselves to finitely generated groups $G$, where being finitely related is equivalent to being finitely presented.

Consider a finitely generated group $G$ and its lattice of quotients $\mathcal{Q}_G$, paying attention at which are finitely presented and which are not. Intuitively, the more types of sublattices $\mathcal{Q}_G$ has, the more sophisticated the internal structure of $G$ will be. We will consider the extreme situation where \emph{all} abstract finite lattices (together with the information at each vertex whether we want that quotient to be finitely related or not) embed in $\mathcal{Q}_G$; we will refer to such a group $G$ as \defin{quotient-saturated}, see \Cref{def: qs} for the details.


The paper is organized as follows: in \Cref{sec: q-sat} we remind the notion of directed acyclic graph (DAG), we use it to establish the formal definition of quotient-saturated group; we also obtain a first partial result (\Cref{prop: rea}) providing the seed for our arguments to obtain finitely presented quotient-saturated groups. In \Cref{sec: main} we prove our main result (\Cref{thm: main}), providing a sufficient condition {(called CEP equipment, see~\Cref{def: CEPeq})} for a group to be quotient-saturated. As a consequence, we deduce that all non-elementary finitely presented subgroups of hyperbolic groups (in particular, non-elementary hyperbolic groups themselves) are quotient-saturated. As a nice corollary, this provides a new obstruction for a group to embed into a hyperbolic group. Finally, in \Cref{sec: acyl} we extend our previous analysis to showing that finitely presented acylindrically hyperbolic groups are all quotient-saturated as well.

Throughout the paper, we adopt the following notation. Given a group $G$, and elements $g,h \in G$, we denote  by $g^h=h^{-1}gh$ the conjugate of $g$ by $h$, and we extend this notation to subsets $R,S\subseteq G$ writing $R^S=\{r^s \mid r\in R,\,\, s\in S\}$. The normal closure of $R\subseteq G$ in $G$ is $\normalcl{R}_G =\gen{R^G}$, namely the smallest normal subgroup of $G$ containing $R$. For any cardinal $n\geq 0$, the free group of rank $n$ is denoted by $\Free[n]$ (or $\Free[A]$ if we want to specify a free basis $A$). 

\section{Quotient-saturated groups}\label{sec: q-sat}

Let $G$ be a finitely generated group and let $N_1,\ldots ,N_n\normaleq G$ be a finite collection of pairwise different normal subgroups of $G$. We can form a directed simple graph (i.e., with no loops and no parallel edges) with set of vertices $V=[n]=\{1,\ldots ,n\}$, and having a directed edge $(i,j)$ from $i$ to $j$ if and only if $N_i<N_j$ (in this note, we use the symbol $<$ to mean strict inclusion). Dually, we can think that vertex $i$ represents the quotient group $G/N_i$, and edge $(i,j)$ represents the proper canonical projection $\pi_{i,j} \colon G/N_i \onto G/N_j$ (proper because the normal subgroups $N_i<N_j$ are different; in other words, we do not consider loops at any vertex representing the identity $\operatorname{id}\colon G/N_i\to G/N_i$). Note that, by construction, this is a directed graph containing \emph{no non-trivial directed closed paths}.

We recall that a directed (simple) graph, or \defin{digraph}, is formally a pair $\Gamma =(V,E)$, where $V=V\Gamma$ is a nonempty set, called the \defin{set of vertices} of $\Gamma$, and $E=E\Gamma$ is a subset of $V\times V$, called the \defin{set of (directed) edges} of $\Gamma$. If $e=(p,q)$ then we say that $e$ is an edge from $p$ to $q$. The \defin{order} of a digraph is its number of vertices. We say that $\Gamma$ is \defin{finite} if both $V$ and $E$ are finite.

A \defin{(directed) path} in a digraph is an alternating sequence of the form $\gamma = v_0 e_1 v_1 e_2 v_2\cdots e_k v_k$, where the $v_i$'s are vertices, the $e_i$'s are edges, and $e_i=(v_{i-1}, v_i)$, for $i=1,\ldots ,k$. Then, we say that $\gamma$ is a path \defin{from $v_0$ to $v_k$} of \defin{length} $k$, namely the number of edges occurring in $\gamma$ (with possible repetitions). Paths of length zero can be identified with vertices, and are called \defin{trivial} paths. If the first and last vertex in a path coincide, we say that it is \defin{closed}. A \defin{directed acyclic graph} (a \emph{DAG}, for short) is a digraph containing no non-trivial closed paths (\aka cycles). Note that we are not requiring any kind of connectivity in the definition of DAG (so, any edgeless graph is, certainly, a~DAG). 

In any DAG $\Gamma$, the existence of a directed path between vertices defines a partial order in $V\Gamma$: for $u,v\in V\Gamma$, $u\leq_{\Gamma} v$ if and only if there is a directed path $v_0 e_1 v_1 e_2 v_2\cdots e_k v_k$ in $\Gamma$ from $u=v_0$ to $v=v_k$. Formally, this partial order is the reflexive-transitive closure of the relation describing adjacency in $\Gamma$. As usual, we write $u<_{\Gamma} v$ to mean $u\leq_{\Gamma} v$ but $u \neq v$. In both notations, $\leq_{\Gamma}$ and $<_{\Gamma}$, we will omit the subscript whenever the underlying DAG $\Gamma$ is clear.

Finally, a \defin{colored DAG} is a 3-tuple $\Gamma =(V,E, c)$ where $(V,E)$ is a DAG and $c\colon V\to \{0,1\}$ is a map assigning a 0/1 color to each vertex. 

\begin{defn}\label{def: qs}
Let $\Gamma =(V,E, c)$ be a colored DAG, and $G$ be a finitely generated group. A \defin{ (quotient) realization} of $\Gamma$ in $G$ is an assignment of a normal subgroup $N_v\normaleq G$ to each vertex $v \in V\Gamma$, 
in such a way that: 
 \begin{itemize}
\item[(i)] for any two distinct vertices $u\neq v$, we have $N_u\neq N_v$;
\item[(ii)] for any two vertices $u,v$, we have $u\leq v$ if and only if $N_u \leqslant N_v$;
\item[(iii)] for any vertex $v$, the quotient group $G_v=G/N_v$ is finitely presented if and only if $c(v)=0$.
 \end{itemize}
When such a realization exists we say that $\Gamma$ is \defin{realizable} in the group $G$, and that $G$ \defin{admits a realization} of $\Gamma$. Finally, a group $G$ is said to be \defin{quotient-saturated} if every finite colored DAG is realizable in $G$. 
\end{defn}

Note that a DAG $\Gamma=(V,E)$ is realizable if and only if its transitive closure is realizable. We emphasize that condition~(ii) is a double implication: for any pair of vertices, $u,v$, it requires that, whenever there is a directed path from $u$ to $v$ then $N_u \leqslant N_v$; but also vice-versa: if there is no directed path from $u$ to $v$ then $N_u$ \emph{must not} be a subgroup of $N_v$. Moreover, in any realization of $\Gamma$, for each vertex $v\in V$, the quotient group $G_v:=G/N_v$ is required to be finitely presented if $c(v)=0$, and not finitely-presented if $c(v)=1$. 

Alternatively, we can think that a realization attaches a quotient $G_v$ of $G$ to each vertex $v\in V$, being finitely presented or not according to the color $c(v)$, and in such a way that non-trivial directed paths represent proper projection maps. Of course, the connection among these dual points of view is $N_v =\ker (G\onto G_v )$. Restricting ourselves to the cases where $G$ is finitely presented, and using the following well-known fact, we can characterize the finite presentability of $G_v=G/N_v$ purely in terms of the normal subgroup $N_v$; this will be convenient later. 


\begin{lem}[Miller III, Thm 2.10 in~\cite{miller_iii_combinatorial_2004}]\label{lem: pres}
Let $\Fn$ be a finitely generated free group, and $R\subseteq \Fn$. If the quotient $\Fn/\normalcl{R}$ is finitely presented then there exists a finite subset $R_0\subseteq R$ such that $\normalcl{R_0}=\normalcl{R}$; in particular, $\Fn/\normalcl{R} =\Fn/\normalcl{R_0}$. \qed
\end{lem}

Clearly, if $G_2$ is a quotient of $G_1$ then every colored DAG realizable in~$G_2$ is also realizable in~$G_1$. In particular, if $G_2$ is quotient-saturated then $G_1$ is quotient-saturated as well. Hence, the question about existence of quotient-saturated groups (of rank $n$) reduces to whether free groups (of rank $n$) are quotient-saturated. In the next section, as a consequence of our main result, it will follow that non-abelian free groups are, in fact, quotient-saturated. At the moment, we can prove a first partial result in this direction: every finite colored DAG is realizable in a free group of big enough rank. To this purpose, we shall use the well-known fact that free factors of free groups, $H\leqff \Free[n]$, satisfy the property $\normalcl{R}_H =H\cap \normalcl{R}_{\Free[n]}$, for every subset $R\subseteq H$ (see \Cref{rem: free factors} below). 

\begin{prop}\label{prop: rea} Any finite colored DAG of order $n$ is realizable in the free group $\FF_{2n}$.
\end{prop}

\begin{proof}
Let us prove it by induction on $n\geq 1$. For $n=1$ the statement is true, since $\FF_2$ admits both finitely presented and not finitely presented quotients. 

Fix $n\geq 2$, suppose the result is true for colored DAG's with less than $n$ vertices, and let $\Gamma$ be a colored DAG with $n$ vertices. By finiteness, $\Gamma$ has a (not necessarily unique) $\leq$-maximal vertex; choose one of them, i.e., a vertex $w$ with no edges going out of it, and remove it from $\Gamma$ together with the edges arriving to it. By induction, the obtained DAG $\Gamma\setminus \{w\}$ admits a realization in $\FF_{2n-2}=\langle x_1, x_2,\ldots ,x_{2n-3},x_{2n-2}\rangle$ by a certain family of normal subgroups $N_{u} \normaleq \FF_{2n-2}, \ u\neq w$. In order to realize $\Gamma$ in $\FF_{2n}=\langle x_1, x_2,\ldots ,x_{2n-1},x_{2n}\rangle=\FF_{2n-2}*\langle x_{2n-1},x_{2n}\rangle$, consider the following assignment of normal subgroups of $\FF_{2n}$ to the vertices of $\Gamma$: 
 \begin{equation}\label{eq: Nprimes}
N'_{u}=\left\{ \begin{array}{ll} \normalcl{N_{u}, x_{2n-1}, x_{2n}}_{\FF_{2n}} &  \text{if } u\neq w\not> u, \\ \normalcl{N_{u}}_{\FF_{2n}} & \text{if } u\neq w> u, \\ \normalcl{x_1, \ldots ,x_{2n-2}, R}_{\FF_{2n}} & \text{if } u=w, \end{array}\right.
 \end{equation}
where $R=\{x_{2n-1}\}$ if $c(w)=0$, and $R=\ker (\langle x_{2n-1}, x_{2n} \rangle \onto L)$ if $c(w)=1$, with $L$ being a two-generated not finitely presented group. In other words, if $G_{u}=\FF_{2n-2}/N_{u}$ are the quotients given by the realization of $\Gamma \setminus \{w\}$ from the induction hypothesis, we are defining $G'_{u}=G_{u}$ if $u\not< w$, $G'_{u}=G_{u}*\langle x_{2n-1}, x_{2n}\rangle$ if $u<w$, and $G'_{w}=\langle x_{2n-1}, x_{2n} \mid R\rangle$ (isomorphic to $\mathbb{Z}$ if $c(w)=0$, and to $L$ if $c(w)=1$).

To see that this is a correct realization for $\Gamma$, let us prove properties (i), (ii), and (iii). Note that, since $\Free_{2n-2}\leqff \Free_{2n}$, we have $N'_{u}\cap \FF_{2n-2}=N_{u}$, for $u\neq w$; so, for any two old vertices $u,v\neq w$, $N'_{u}=N'_{v}$ implies $N_{u}=N_{v}$ which implies $u=v$. Also, if $u\neq w\not> u$ we have $x_{2n}\in N'_{u}$ while $x_{2n}\not\in N'_{w}$ (obviously if $c(w)=0$, and because $L$ would be cyclic and so finitely presented if $c(w)=1$ and $x_{2n} \in N'_{w}$); and if $u\neq w>u$, $N'_{u}$ contains no non-trivial word from $\gen{x_{2n-1}, x_{2n}}$, while $N'_{w}$ always does. Therefore, $N'_{w}\neq N'_{u}$ and (i) follows. 

To see (ii), take two old vertices $u,v\neq w$: by hypothesis, if $u\leq v$ then $N_{u} \leqslant N_{v}$ and so, $N'_{u} \leqslant N'_{v}$ (note that $u\neq w\not>u$ and $v\neq w>v$ cannot happen simultaneously); conversely, if $N'_{u} \leqslant N'_{v}$ then, intersecting with $\FF_{2n-2}$, we have $N_{u} \leqslant N_{v}$ and so, $u\leq v$. Moreover, if $u\neq w\not> u$ then $x_{2n}\in N'_{u}$ but $x_{2n}\not\in N'_{w}$ so, $N'_{u}\not\leqslant N'_{w}$; and if $u\neq w>u$ then $N'_{u}=\normalcl{N_{u}}_{\FF_{2n}}\leqslant \normalcl{x_1,\ldots ,x_{2n-2}}_{\FF_{2n}}< N'_{w}$. This completes the prove of (ii).

Finally, (iii) is clear: for any old vertex $u\neq w$, $G'_{u}=\FF_{2n}/N'_{u}$ is isomorphic to either $G_{u}$ (if $u\neq w\not>u$) or to $G_{u}*\langle x_{2n-1}, x_{2n}\rangle$ (if $u\neq w > u$); in both cases, $G'_{u}$ is finitely presented if and only if $G_{u}$ is so, which, by induction, happened if and  only if $c(u)=0$. On the other hand, by construction, $G'_{w}=\langle x_{2n-1}, x_{2n} \mid R \rangle$ is isomorphic to $\mathbb{Z}$, if $c(w)=0$, and to $L$ if $c(w)=1$; hence, it is finitely presented (in fact, infinite cyclic) if and only if $c(w)=0$. 

Thus, the normal subgroups defined in \eqref{eq: Nprimes} realize $\Gamma$ in $\FF_{2n}$, completing the induction step.  
\end{proof}

\begin{rem}
As a corollary of \Cref{prop: rea}, one is tempted to deduce that the free group of countable rank, $\FF_{\aleph_0}=\FF_{\{a_1, a_2,\ldots\}}$, is quotient-saturated. This intuition is correct, but there is a subtle detail to highlight here. We can generalize \Cref{def: qs} to non finitely generated ambient groups $G$ in two different ways: taking condition (iii) verbatim, or replacing it by: 
\begin{itemize}
\item[(iii')] for any vertex $v$, the quotient group $G_v=G/N_v$ is finitely related if and only if $c(v)=0$.
\end{itemize}
Both are equivalent when $G$ is finitely generated, but not equivalent otherwise. An easy adaptation of the proof of \Cref{prop: rea} shows that $\Free_{\aleph_0}$ is quotient-saturated in both senses: given a finite $DAG$ $\Gamma$ of order $|V\Gamma|=n$, realize $\Gamma$ in $\Free_{2n}\leqff \Free_{\aleph_0}$ by certain normal subgroups $N_v \normaleq \Free_{2n}$, $v\in V\Gamma$, and then replace them by $N'_v=\normalcl{N_v, x_{2n+1}, x_{2n+2},\ldots }_{\Free_{\aleph_0}} \normaleq \Free_{\aleph_0}$. Since each quotient $\Free_{\aleph_0}/N'_v$ is finitely generated, it will be finitely presented if and only if it is finitely related (which happens if and only if $c(v)=0$). For a general non finitely generated group $G$, the notions of realizability and  quotient-saturation may not coincide. 
\end{rem}

The main idea in the proof of \Cref{prop: rea} is to iteratively extend the realization from smaller to bigger DAG's, vertex by vertex, by using two fresh ambient generators to accommodate each new vertex. This has the inconvenient of using a number of ambient generators growing with the order of the DAG, $|V\Gamma|$. With this approach, it is not clear how to find \emph{finitely presented} groups $G$ being quotient-saturated: with a bounded number of ambient generators, in order to realize big DAG's, we will have to introduce infinite sets of relations to get the required not finitely-presented quotients, and then we will not be able to kill the involved generators for the subsequent quotients; therefore, we will have to mandatorily take those ugly relations into account at \emph{all} the subsequent vertices from $\Gamma$, no matter if the corresponding quotients must be finitely presented or not. 

Instead, in the next section, we will adopt a different strategy. In order to provide a large family of examples of finitely presented quotient-saturated groups, we will make use of a strong result by Olshanskii. 

\section{The main result}\label{sec: main}

Let us pay closer attention to the technical property we have used in the proof of \Cref{prop: rea} about normal closures. In general, whenever we have an extension of groups $H\leqslant G$, and a subset $R\subseteq H$, the inclusion $\normalcl{R}_H\leqslant H\cap \normalcl{R}_G$ always holds, while equality is not necessarily true. Not even in free groups: take, for example, $G=\FF_{\!\set{a,b}}$ the rank two free group with basis $\set{a,b}$, $H=\gen{a, b^{-1}ab}$ and $R=\{b^{-1}ab\}$; since~$H$ is again a rank two free group (this time with basis $\{a, b^{-1}ab\}$) we have that $\normalcl{R}_H =\gen{a^{-n}b^{-1}aba^n,\, n\in \mathbb{Z}}<H$, whereas $H\cap \normalcl{R}_{G}=H\cap \normalcl{a}_G=H$. This property about normal closures was already considered in the literature for other purposes. 

\begin{defn} \label{def: CEP}
An extension of groups, $H\leqslant G$, is said to satisfy the \defin{Congruence Extension Property} (is \defin{CEP}, for short) if, for every subset $R\subseteq H$, $\normalcl{R}_H =H\cap \normalcl{R}_G$.
\end{defn}

The proper way to understand the CEP property is as follows: since $\normalcl{R}_H \leqslant \normalcl{R}_G$, the natural inclusion $\iota \colon H\hookrightarrow G$ descends to a well-defined group homomorphism $\overline{\iota}\colon H/\normalcl{R}_H \to G/\normalcl{R}_G$, making the corresponding diagram commutative:
 $$
\begin{tikzcd}
H\arrow[r,"\iota", hook] \arrow[d,"\pi_H"', two heads] \arrow[rd,phantom, "\scriptstyle{///}" description] & G\arrow[d,"\pi_G", two heads] \\
H/\normalcl{R}_H \arrow[r,"\overline{\iota}"] & G/\normalcl{R}_G      
\end{tikzcd}
 $$
Note that $\overline{\iota}$ is injective if and only if $\normalcl{R}_H =\ker(\pi_H)=\ker(\pi_H \overline{\iota})=\ker(\iota \pi_G )$, i.e., if and only if
 \begin{equation} \label{eq: key property}
\normalcl{R}_H =H\cap \normalcl{R}_G.
 \end{equation}

It is straightforward to see that the property of being CEP is transitive: for $H\leqslant K\leqslant G$, if $H\leqslant K$ is CEP and $K\leqslant G$ is CEP, then $H\leqslant G$ is CEP as well. Moreover, if $H\leqslant G$ is CEP then $H\leqslant K$ is CEP as well.

\begin{rem} \label{rem: free factors}
As mentioned above, free factors in free groups are CEP: if $\Free[n]=L*H$ then, for every $R\subseteq H$, $\Free[n]/\normalcl{R}_{\Free[n]} =L*H/\normalcl{R}_H$ and the above homomorphism $\overline{\iota}$, sending $h\normalcl{R}_H$ to $h\normalcl{R}_{\Free[n]}$, is clearly injective. Olshanskii proved in \cite{olshanskij_sq-universality_1995} (see~\Cref{thm: Ols} below) that the CEP property shows up also in other situations far from free factors: for any hyperbolic group $G$, and any non-elementary subgroup $D\leqslant G$, one can find a free subgroup $H\leqslant D\leqslant G$ of any desired rank $n\in [2,\aleph_0]$ such that $H\leqslant D$ is CEP. This will be a crucial ingredient to prove our main result. 
\end{rem}

\begin{thm}[Olshanskii, Thm. 3 in~\cite{olshanskij_sq-universality_1995}]\label{thm: Ols}
For any hyperbolic group $G$, any non-elementary sub\-group $D\leqslant G$, and any $2\leq n\leq \aleph_0$, there exists a rank $n$ free subgroup $H\leqslant D\leqslant G$ such that $H\leqslant G$ is CEP (and so, $H\leqslant D$ is CEP).\qed
\end{thm}

As a consequence observe that, for a group $G$, the existence of non-abelian free subgroups $\Free[n]\simeq H\leqslant G$ such that $H\leqslant G$ is CEP does not depend on the specific rank $2\leq n\leq \aleph_0$: certainly, if $G$ contains a CEP subgroup $H\leqslant G$ isomorphic to $\Free[n]$, we can always select a free factor of rank two, $\Free[2]\leqff \Free[n]\simeq H\leqslant G$ and, by transitivity, it will be a free CEP subgroup of $G$ of rank two. And, conversely, if $G$ contains a free CEP subgroup of rank two, $\Free[2]\simeq H\leqslant G$, applying Olshanskii \Cref{thm: Ols}, we can select inside $H$ a free CEP subgroup of any desired rank $2\leq n\leq \aleph_0$, $\Free[n]\leqslant \Free[2]\simeq H\leqslant G$ which, again by transitivity, will also be CEP as a subgroup of $G$. This group property turns out to be sufficient for a finitely presented group $G$ to be quotient saturated: this is our main result. 

\begin{defn}\label{def: CEPeq}
A group $G$ is \defin{CEP-equipped} if it contains free non-abelian CEP subgroups (of some, and hence any, rank $2\leq n\leq \aleph_0$). 
\end{defn}

\begin{thm}\label{thm: main}
Let $G$ be a finitely presented group. If $G$ is  \CEPeq\  then $G$ is quotient-saturated.
\end{thm}

\begin{proof}
Let $G\simeq \langle X \mid S\,\rangle=\FF(X)/\normalcl{S}_{\FF(X)}$ be a finite presentation for $G$, i.e., $X$ is a finite alphabet, $\FF(X)$ is the free group on $X$, and $S\subseteq \FF(X)$ is a finite set of relations. Denote by $\overline{\phantom{a}}\colon \FF(X)\onto G$, $w\mapsto \overline{w}$, the quotient map.

Fix a finite colored DAG, $\Gamma =(V,E,c)$, write $n=2|V|$, and let us find a realization of $\Gamma$ in the group $G$. By hypothesis, $G$ contains a rank $n$ free CEP subgroup $\Free[n]\simeq H\leqslant G$. And, by \Cref{prop: rea}, $\Gamma$ admits a realization in $H$, say, by some normal subgroups $N_v=\normalcl{\overline{R_v}}_H \normaleq H\leqslant G$, $v\in V$, where $R_v \subseteq \FF(X)$ are certain sets of words with $\overline{R_v}\subseteq H\leqslant G$; moreover, by \Cref{lem: pres} and passing to an appropriate finite subset when necessary, we can assume $|R_v|<\infty$ whenever $H_v=H/N_v$ is finitely presented, i.e., whenever $c(v)=0$. 

Let us replicate this realization, but using the appropriate normal subgroups in $G$ instead of normal subgroups in $H$. Attach to every vertex $v\in V$ the subgroup $N'_v =\normalcl{\overline{R_v}}_G \normaleq G$. Clearly, if $u\neq v$ then $\normalcl{\overline{R_u}}_H =N_u \neq N_v =\normalcl{\overline{R_v}}_H$ and hence $N'_u =\normalcl{\overline{R_u}}_G \neq \normalcl{\overline{R_v}}_G =N'_v$ (recall that we have $\normalcl{\overline{R_u}}_H =H\cap \normalcl{\overline{R_u}}_G$ and $\normalcl{\overline{R_v}}_H =H\cap \normalcl{\overline{R_v}}_G$, because $H\leqslant G$ is CEP). This shows property~(i). 

If $u\leq v$ then $N_u \leqslant N_v$ and hence $N'_u =\normalcl{\overline{R_u}}_G =\normalcl{N_u }_G \leqslant \normalcl{N_v }_G =\normalcl{\overline{R_v}}_G =N'_v$. Conversely, if $N'_u \leqslant N'_v$ then, intersecting with $H$, we get $N_u \leqslant N_v$ and so, $u\leq v$. This shows property~(ii).

Finally, we claim that the quotient $G_v =G/\normalcl{\overline{R_v}}_G$ is finitely presented if and only if $H_v =H/\normalcl{\overline{R_v}}_H$ is so. In fact, if $H_v$ is finitely presented then, by construction, $|R_v|<\infty$ and 
 $$
G_v=G/\normalcl{\overline{R_v}}_G \simeq \FF(X)/\normalcl{S\cup R_v}_{\FF(X)}
 $$
is finitely presented as well, since $|S\cup R_v|\leqslant |S|+|R_v|<\infty$. Conversely, suppose that $G_v \simeq \FF(X)/\normalcl{S\cup R_v}_{\FF(X)}$ is finitely presented. By \Cref{lem: pres}, there is a finite subset $T_v \subseteq S\cup R_v\subseteq \FF(X)$ such that $\normalcl{T_v}_{\FF(X)}= \normalcl{S\cup R_v}_{\FF(X)}$. Now, projecting this subgroup equality down to $G$, we get $\normalcl{\overline{T_v}}_G =\normalcl{\overline{S}\cup \overline{R_v}}_G =\normalcl{\overline{R_v}}_G$. Finally, note that the possible elements from $T_v$ which belong to $S$ vanish in $\overline{T_v}$ so, $\overline{T_v}\subseteq \{1\}\cup \overline{R_v}\subseteq H$ and 
 $$
\normalcl{\overline{T_v}}_H =H\cap \normalcl{\overline{T_v}}_G =H\cap \normalcl{\overline{R_v}}_G =\normalcl{\overline{R_v}}_H.
 $$
Therefore, $H_v =H/\normalcl{\overline{R_v}}_H =H/\normalcl{\overline{T_v}}_H$ is finitely presented as well. This shows property~(iii) and completes the proof that the above $N'_v$'s determine a realization of $\Gamma$ in $G$. Since this is valid for every finite colored DAG $\Gamma$, the group $G$ is quotient-saturated.
\end{proof}

Using Olshanskii \Cref{thm: Ols}, we obtain the first known explicit examples of finitely presented quotient-saturated groups. 

\begin{cor}\label{cor: hyp}
Any non-elementary finitely presented subgroup $D$ of a hyperbolic group $G$ (and hence any non-elementary hyperbolic group itself) is quotient-saturated.
\end{cor}

\begin{proof}
By Olshanskii \Cref{thm: Ols}, there exists a free subgroup of any desired  rank $2\leq n\leq \aleph_0$, $H\leqslant D\leqslant G$ such that $H\leqslant D$ is CEP, therefore, $D$ is CEP-equipped. Now, by \Cref{thm: main}, $D$ is quotient-saturated.
\end{proof}

Of course, as particular cases, non-abelian free groups $\FF_{n}$ of all possible ranks $n\geq 2$ are also quotient-saturated. 
%

\begin{rem}

Obviously, if $G$ admits $\Free[2]$ as a quotient then it is quotient-saturated. We remark that the converse is far from true: even groups having neither $\FF_2$ nor even $\FF_1=\mathbb{Z}$ as a quotient may be quotient-saturated: perfect hyperbolic groups, for example; one can easily construct such groups by writing a finite presentation with a set of relations satisfying $C'(1/6)$ (so, presenting a hyperbolic group), but having trivial abelianization, $G^{\rm ab}=1$.
\end{rem}

\medskip

Finally, we highlight the following corollary providing a new obstruction for a group $D$ to embed into a hyperbolic group; we think this could be of independent interest. 

\begin{cor} \label{cor: }
Non-elementary, finitely presented, non quotient-saturated groups $D$ do not embed in any hyperbolic group $G$. \qed
\end{cor}

In particular, non-elementary finitely presented just infinite groups (including infinite simple groups) do not embed in hyperbolic groups. 

\section{Acylindrically hyperbolic groups are quotient saturated}\label{sec: acyl}

In this final section, we will further leverage our main result to extend the realm of quotient-saturated groups beyond non-elementary finitely presented subgroups of hyperbolic groups. 

The profound significance of hyperbolic groups, introduced by Gromov in 1987, has generated several attempts at generalization. In the 2010's, the family of acylindrically hyperbolic groups was introduced by Denis Osin in \cite{osin_acylindrically_2016}: although greatly extending the family of hyperbolic groups, one can prove that acylindrically hyperbolic groups still satisfy lots of properties with hyperbolic flavour. As we will see in this section, CEP-equipment is one of them. 

An action by isometries of a group $G$ on a metric space $S$ is called \textit{acylindrical} if, for every~$\epsilon >0$, there exist $R,N>0$ such that, for every two points $x,y\in S$ with $d(x,y)\geq R$, we have the inequality $|\{ g\in G \mid d(x,gx)\leq \epsilon \text{ and } d(y,gy)\leq \epsilon \}|\leq N$ (this definition was introduced by Bowditch in~\cite{bowditch_tight_2008}, while the notion of acylindricity for actions on trees dates back to Sela~\cite{sela_acylindrical_1997}). A group $G$ is \textit{acylindrically hyperbolic} if it admits a non-elementary acylindrical action on a hyperbolic space (recall that an action of a group $G$ on a hyperbolic space $S$ is \textit{elementary} if the limit set of $G$ on $\partial S$ contains at most $2$ points). We refer the reader to~\cite{osin_acylindrically_2016} for an introduction to acylindrically hyperbolic groups. This is a family of groups greatly extending that of non-elementary hyperbolic groups. A list of prominent examples of acylindrically hyperbolic groups follows (see~\cite{osin_acylindrically_2016} for more details and references): 
 \begin{itemize}
\item[(a)] non-elementary 
subgroups of
hyperbolic groups;
\item[(b)] non-virtually-cyclic relatively hyperbolic groups with proper peripheral subgroups; 
\item[(c)] the mapping class group $\operatorname{MCG}(\Sigma_{g,p})$ of a closed surface of genus $g$ with $p$ punctures, except for the cases $g=0$ and $p\leq 3$ (when these groups are finite);
\item[(d)]  $\operatorname{Out}(\FF_n)$ for $n\geq 2$;
\item[(e)] every right-angled Artin group which is neither cyclic nor directly decomposable; 
\item[(f)] $1$-relator groups with 3 or more generators;
 \end{itemize}

The goal of the present section is to prove that every finitely presented acylindrically hyperbolic group is quotient-saturated. This clearly extends \Cref{cor: hyp} since many significant subfamilies of groups in the above list contain $\mathbb{Z}^2$ and so, they do not embed into any hyperbolic group. Our goal will follow from a few known results about acylindrically hyperbolic groups, plus a property known to experts which we prove here since it seems not to be included in the existing literature. 

Dahmani--Guirardel--Osin introduced the notion of a hyperbolically embedded subgroup $H$ in a group $G$, denoted $H\hookrightarrow_{h} G$ (the precise definition is a bit technical and not relevant for our purposes; see \cite[Def.~2.1]{dahmani_hyperbolically_2016}); it turned out to be equivalent to acylindrical hyperbolicity in the following terms:

\begin{thm}[{\cite[Thm.~1.2]{osin_acylindrically_2016}}] \label{thm: Osin acy-hyp-emd}
For any group G, the following conditions are equivalent: 
 \begin{itemize}
\item[(a)] $G$ is acylindrically hyperbolic;
\item[(b)] $G$ contains a proper infinite hyperbolically embedded subgroup, $H\hookrightarrow_{h} G$. \qed
 \end{itemize}
\end{thm}

Relating these two notions, they showed the following two results, clearly connecting with our interests. 

\begin{thm}[{\cite[Thm.~2.24]{dahmani_hyperbolically_2016}}] \label{thm: DGO ndg hyp emb}
Let $G$ be an acylindrically hyperbolic group. Then, for any ${n\in [2,\infty)}$, $G$ contains a hyperbollically embedded subgroup $\Fn \times K \isom H \hookrightarrow_{h} G$, where $|K| <\infty$. \qed
\end{thm}

\begin{defn}
An extension of groups, $H\leqslant G$, is said to be \defin{almost-CEP} if there exists a finite subset $S\subseteq H \setminus\set{\trivial}$ such that, for every $R\subseteq H$ with $S\cap \normalcl{R}_H=\emptyset$, we have $\normalcl{R}_H =H\cap \normalcl{R}_G$.
\end{defn}

\begin{thm}[{\cite[Thm.~2.27(a)]{dahmani_hyperbolically_2016}}] \label{thm: DGO almost CEP}
Let $G$ be a group. Every hyperbolically embedded subgroup $H\hookrightarrow_{h} G$ is almost-CEP. \qed
\end{thm}

The following improvement of \Cref{thm: DGO almost CEP} is known to experts but does not seem to be present in the literature. We include a proof for completeness.  

\begin{prop}\label{prop: Osin}
Acylindrically hyperbolic groups $G$ are CEP-equipped.
\end{prop}

\begin{proof}
It suffices to find a free rank 2 CEP subgroup of $G$.  

By \Cref{thm: DGO ndg hyp emb} with $n=2$, and \Cref{thm: DGO almost CEP}, we have a subgroup $H\leqslant G$ such that $H=H'\times K$ with $H'\simeq \Free[2]$, $|K|<\infty$, and with $H\leqslant G$ being almost-CEP, i.e., there is a finite subset $S\subseteq H \setminus \set{\trivial}$ satisfying $\normalcl{R}_H =H\cap \normalcl{R}_G$, for every $R\subseteq H$ with $S\cap \normalcl{R}_H=\emptyset$. Now, since $H'$ is residually finite, it contains a normal finite index subgroup $N\normaleq H'$ such that $S\cap N=\emptyset$. And, by Olshanskii's \Cref{thm: Ols}, there is $M\leqslant N$ such that $M\simeq \Free[2]$ and $M\leqslant H'$ is CEP.  

Now, we claim that the extension $M\leqslant G$ is CEP. In fact, we have $M\leqslant N\normaleq H'\leqslant H'\times K=H\leqslant G$ and, for every $R\subseteq M$, $\normalcl{R}_{H}=\normalcl{R}_{H'}\leqslant N$ and so, $S\cap \normalcl{R}_{H}=\emptyset$. Therefore, $\normalcl{R}_H =H\cap \normalcl{R}_G$ and hence, $\normalcl{R}_{M}=M\cap \normalcl{R}_{H'}=M\cap \normalcl{R}_H =M\cap H\cap \normalcl{R}_G =M\cap \normalcl{R}_G$. This proves that $M\leqslant G$ is CEP, concluding the proof. 
\end{proof}

Our final claim follows directly from \Cref{prop: Osin} and \Cref{thm: main}.

\begin{thm}\label{thm: acyl hyp are qs}
Finitely presented acylindrically hyperbolic groups are quotient-saturated. \qed
\end{thm}





\section*{Acknowledgements}

We are very grateful to Ashot Minasyan for bringing to our attention the CEP property, which allowed to significantly improve the scope of our first draft on quotient-saturation. We also thank Montse Casals for suggesting us the possible expansion to acylindrically hyperbolic groups, to Denis Osin for sharing with us a proof for \Cref{prop: Osin}, and to Ashot Minasyan for suggesting a convenient variation. 

\noindent The first named author thanks the support from the Universitat Politècnica de Catalunya (UPC) through a María Zambrano grant. The second named author thanks the support from UPC through a Margarita Salas grant. The three authors acknowledge support from the Spanish Agencia Estatal de Investigaci\'on through grant PID2021-126851NB-100 (AEI/ FEDER, UE).

\renewcommand*{\bibfont}{\small}
\printbibliography

\end{document}